\documentclass[12pt]{article}
\usepackage{lingmacros}
\usepackage{amssymb}
\usepackage{amsmath}
\usepackage{geometry}
\usepackage{setspace}
\usepackage{amsthm}
\usepackage{mathrsfs}
\usepackage{changepage}
\usepackage{listings}
\usepackage{pdflscape}
\usepackage{graphicx}
\usepackage{mathrsfs}
\usepackage[hidelinks]{hyperref}
\DeclareMathOperator{\C}{\mathbb{C}}

\DeclareMathOperator{\F}{\mathbb{F}}

\DeclareMathOperator{\Q}{\mathbb{Q}}
\DeclareMathOperator{\Z}{\mathbb{Z}}
\DeclareMathOperator{\sF}{\mathcal{F}}

\theoremstyle{plain}
\newcounter{thmsection}

\newtheorem{thm}{Theorem}[thmsection]
\newtheorem{prop}[thm]{Proposition}
\newtheorem{prop-example}[thm]{Proposition-Example}
\newtheorem{cor}[thm]{Corollary}
\newtheorem{lem}[thm]{Lemma}

\theoremstyle{definition}
\newtheorem{defn}[thm]{Definition}
\newtheorem{example}[thm]{Example}
\newtheorem{remark}[thm]{Remark}

\newtheorem{obs}[thm]{Obstruction}

\newcommand{\fakesection}[1]{
  \par\refstepcounter{section}
  \sectionmark{#1}
  \addcontentsline{toc}{section}{\protect\numberline{\Roman{section}}\: #1}
 \begin{center}
 \textsc{\Roman{section} - #1}
 \end{center}
 \setcounter{thm}{0}
\stepcounter{thmsection}
}
\newcommand{\fakesectionnonumber}[1]{
  \par\refstepcounter{section}
  \sectionmark{#1}
  \addcontentsline{toc}{section}{#1}
 \begin{center}
 \textsc{#1}
 \end{center}\normalsize
 \setcounter{section}{1}
 \setcounter{thm}{0}
}

\newcommand\blfootnote[1]{
  \begingroup
  \renewcommand\thefootnote{}\footnote{#1}
  \addtocounter{footnote}{-1}
  \endgroup
}
\begin{document}
\begin{center}
\textbf{REPRESENTATION RINGS OF FUSION SYSTEMS}

\textbf{AND BRAUER CHARACTERS}  
\end{center}
\vspace{0.2cm}
\begin{center}
\textsc{Thomas Lawrence}
\end{center}
\vspace{0.2cm}
\begin{list}{}{\leftmargin = 1.5cm \rightmargin = 1.5cm}
\item{
\small A\textsc{bstract}. Let $\sF$ be a fusion system over a $p$-group $S$. We study the complex character ring $R_{\C}(\sF)$ of $\sF$ by applying techniques from modular character theory to $\sF$-stable characters. We use these techniques to investigate a conjecture posed by Jason Semeraro concerning the volume of $R_{\C}(\sF)$ as a $\Z$-lattice. Proving it holds for all saturated fusion systems would allow for easy verification that a given set of linearly independent $\sF$-stable characters forms a $\Z$-basis of $R_{\C}(\sF)$. We prove that this conjecture holds for all non-exotic fusion systems and a weakened conjecture holds for all fusion systems. We also show that any minimal counter example must be indecomposable by describing the characters of a product of two fusion systems. As a byproduct of our proof method, we describe the modular character rings of $\sF$, provide analogues of the decomposition and Cartan matrices for $\sF$-stable characters, and give a method for decomposing the regular character of $S$ into $\sF$-stable constituents.
\normalsize}
\end{list}
\vspace{0.2cm}
\fakesectionnonumber{Introduction}
\vspace{0.2cm}
\blfootnote
{\scriptsize
\hspace*{-0.4cm}\hspace*{-0.63cm}\textit{Date:} \today \newline
\hspace*{-0.4cm}\textit{Keywords:} Fusion Systems, Modular Character Theory, Volumes of Lattices \newline
\hspace*{-0.4cm}\textit{2010 Mathematics Subject Classification:} 20C20, 20D20, 20C99\newline
\hspace*{-0.4cm}\textit{Acknowledgements:} We thank Jason Semeraro for posing the original problem and numerous valuable discussions and Benjamin Sambale for his helpful comments on a previous version of this preprint.
}
\normalsize
A fusion system is a category designed to capture the local structure of a finite group at a prime. In more detail, a fusion system $\sF$ over a $p$-group $S$ is a category with objects $\text{Ob}(\sF) = \{P \colon P \leq S\}$ and morphism sets $\mathrm{Hom}_{\sF}(P, Q)$ consisting of injective group homomorphisms designed to mimic conjugation. For example, given a finite group $G$ and $S \leq G$ a $p$-subgroup, then the fusion system \emph{realised by $G$}, written $\sF_S(G)$, is the category with $\mathrm{Hom}_{\sF}(P, Q) = \mathrm{Hom}_G(P,Q) := \{c_g|_P \colon P^g \leq Q\}$, where $c_g$ is the conjugation map induced by $g$. The loosest definition of a fusion system is too general for most purposes and it is standard to instead work with \emph{saturated} fusion systems, which are fusion systems satisfying conditions ensuring that they behave like the fusion system $\sF_S(G)$ with $S \in \text{Syl}_p(G)$. When $\sF = \sF_S(G)$ with $S \in \text{Syl}_p(G)$ for some finite group $G$ we say that $\sF$ is \emph{non-exotic}. This terminology correctly indicates that there exist \emph{exotic} fusion systems, which are saturated fusion systems which are not of the form $\sF_S(G)$ for any finite group $G$ and $S \in \text{Syl}_p(G)$.

There has been recent interest (\cite{completion theorem}, \cite{unique factorisation}, \cite{symmetric groups}, \cite{sambale}, among others) in exploring the ring $R_{\C}(\sF)$ of virtual complex $\sF$-stable characters of $S$, which by Proposition 3.11 of \cite{aycin thesis} is the Grothendieck completion of the semiring $R_{\C}^+(\sF)$ of complex $\sF$-stable characters. 

An element $\chi \in R_{\C}^+(\sF)$ is \emph{$\sF$-indecomposable} if it is an atomic element of the additive monoid of $R_{\C}^+(\sF)$. The set $\text{Ind}(\sF)$ of $\sF$-indecomposable characters, unlike the character theory of finite groups, need not be linearly independent; $\sF$-stable characters need not admit a \emph{unique} decomposition into $\sF$-indecomposable constituents (see Example A.2 in  \cite{unique factorisation}). Because of this fact, it is difficult to determine if a given set of $\sF$-indecomposable characters (or any set of virtual $\sF$-stable characters, for that matter) forms a $\Z$-basis of $R_{\C}(\sF)$. It is with this problem in mind that we formulate the main conjecture of this paper: given a $\Z$-basis $B$ of $R_{\C}(\sF)$ we may define the \emph{$\sF$-character table with respect to $B$}: $X_{B}(\sF) := (\chi(x))_{\chi \in B, x^{\sF} \in \text{cl}(\sF)}$. Based on unpublished calculations, Jason Semeraro has conjectured the following:
\newline\newline
\textbf{Conjecture A} (Semeraro)\textbf{.}
Let $\sF$ be a fusion system on a $p$-group $S$ and $B$ be any $\Z$-basis of $R_{\C}(\sF)$. Writing $\text{fccl}(\sF)$ for a set of fully $\sF$-centralised conjugacy class representatives we have
\[|\mathrm{det}(X_B(\sF))|^2 = \displaystyle\prod_{x \in \text{fccl}(\sF)} |C_S(x)|.\]
\newline\newline
This invariant is independent of our choice of basis (c.f. Lemma \ref{basis invariant lemma}) and in fact characterises $\Z$-bases of $R_{\C}(\sF)$ via elementary lattice theory, where this determinant corresponds to the square of the volume of $R_{\C}(\sF)$ (viewed as a $\Z$-lattice embedded in $\C^{|\text{cl}(\sF)|}$). Hence, if this conjecture holds, we can verify if some $B \subseteq R_{\C}(\sF)$ is a basis by computing the determinant of $X_B(\sF)$.

In this paper we show that this conjecture holds in a large variety of examples:
\newline\newline\textbf{Theorem A} (Proposition \ref{transitive prop}, Theorem \ref{conj a for non-exotic}, Corollary \ref{exotic p power det}, Theorem \ref{min counter-example})\textbf{.}
\textit{Conjecture A holds when $\sF$ is transitive or $\sF = \sF_S(G)$ with $S \in \text{Syl}_p(G)$ for some finite group $G$. Additionally, if Conjecture A holds for $\sF_1, \sF_2$ then it holds for $\sF_1 \times \sF_2$.}
\newline\newline
We also prove the following weaker result that holds for all fusion systems, regardless of saturation:\newline
\newline\textbf{Theorem B} (Corollary \ref{exotic p power det})\textbf{.}
\emph{Given any fusion system $\sF$ on a $p$-group $S$, $|\mathrm{det}(X_B(\sF))|^2$ is a power of $p$.}
\newline\newline
To prove these theorems we exploit a symmetry between $\sF$-stable characters for $\sF$ realised by a group $G$ and Brauer characters (characters of modular representations) of $G$.
For some prime $p$ we will write the \emph{$p$-Brauer characters} of $G$ to refer to the characters of $p$-modular representations of $G$. We omit $p$ if the specific prime is unimportant. We refer the reader to Chapter 2 of \cite{navarro} for the necessary background on these objects in the context of modular representation theory. 

In some heuristic sense the $\sF$-stable characters behave as if they were the $p'$-Brauer characters of some group $G$. This connection is made clear when $|G| = p^aq^b$ and $\sF$ is the fusion system realised by $G$ over a Sylow $p$-subgroup $S$: the $\sF$-indecomposable characters are exactly the restrictions to $S$ of irreducible $q$-Brauer characters of $G$ (see Proposition \ref{two prime factors}). 

We may frame these results as a generalisation of the theory of $\pi$-partial characters (see Chapter 3 in \cite{isaacs}) in the case that $\pi$ is every prime dividing $|G|$ except $p$. The theory in \cite{isaacs} only holds for $\pi$-separable groups, but our methods will work for any group. After the release of the original version of this preprint the author was informed by Benjamin Sambale that these constructions are very reminiscent of unpublished work by Olsson in \cite{olsson}.

As part of our proof of Theorem A we describe the ring of $\sF$-stable Brauer characters:\newline\newline
\textbf{Theorem C} (Proposition \ref{mod p prop}, Theorem \ref{mod rank theorem}, Corollary. \ref{b lin indep})\textbf{.}
\emph{Let $\sF$ be any fusion system on a $p$-group $S$, $\ell$ a prime. Let $R_{\overline{\F}_{\ell}}(\sF)$ be the ring of $\sF$-stable $\ell$-Brauer characters and $\mathcal{M}_{\ell}$ the maximal ideal of the algebraic integers $\mathcal{R}$ containing $\ell$. If $B$ is a $\Z$-basis for $R_{\C}(\sF)$, then $\pi_{\ell}(B)$ is a $\F_p$-basis for $R_{\ell}(\sF)$ }
\newline\newline
We use this theorem to prove Theorem A by utilising the isomorphism between representation rings to show that the rows of $X_B(\sF)$ are linearly independent mod $\mathcal{M}_{\ell}$ for all primes $\ell \neq p$.
Combined with the fact that $|\mathrm{det}(X_B(\sF))|^2 \in \Z$ (see Corollary \ref{det 2}), we have that $|\mathrm{det}(X_B(\sF))|^2$ is a power of $p$.
The refinement of this result to the full statement of Conjecture A for non-exotic fusion systems is where the parallels with modular character theory are fully utilised. 

As a byproduct of our proof method, we are able to elaborate on a recent conjecture (Conjecture 2.18 in \cite{symmetric groups}), which asks if all $\sF$-indecomposable characters appear as subcharacters of the regular character $\rho_S$ of $S$. 
This was shown to be false in \cite{sambale}, but we prove that the decomposition of $\rho_S$ with respect to a $\Z$-basis $B$ of $R_{\C}(\sF)$ is still well behaved:\newline\newline
\textbf{Theorem D} (Proposition \ref{regular character})\textbf{.}
\emph{
Let $\rho_S$ denote the regular character of $S$. Let $\sF = \sF_S(G)$ for some $G$ and $B$ a basis for $R_{\C}(\sF)$. For $\chi \in \text{Irr}(G)$ we write the decomposition of $\mathrm{Res}^G_S(\chi)$ over $B$ as $\sum_{\psi \in B} d^{B,G}_{\chi\psi}\psi$ for some $d^{B,G}_{\chi\psi} \in \Z$. 
Let $\Phi^{B,G}_{\psi} := \sum_{\chi \in \text{Irr}(G)}d^{B,G}_{\chi\psi}\chi \in R_{\C}(G)$, then we have
\[\rho_S = \sum_{\psi \in B} \frac{\Phi^{B,G}_{\psi}(1)}{[G:S]}\psi\]
and these coefficients are integers.
}
\setcounter{section}{0}
\fakesection{A brief overview of $\sF$-stable characters} 
\begin{defn}
Let $G$ be a group and $H, K \leq G$. For $g \in G$ we write $c_g$ for the automorphism of $G$ induced by conjugation by $g$. If we let
\[\mathrm{Hom}_G(H, K) := \{\phi \in \mathrm{Hom}(H, K) \colon \phi = c_g|_H \text{ for some } g \in G\},\]
then we may define a \emph{fusion system $\sF$ over a $p$-group $S$} to be a category with $\text{Ob}(\sF) = \{P \colon P \leq S\}$ and morphism sets $\mathrm{Hom}_{\sF}(P, Q)$ consisting of injective group homomorphisms such that 
\[\mathrm{Hom}_S(P,Q) \subseteq \mathrm{Hom}_{\sF}(P,Q) \quad \forall P, Q \leq S.\] 
Furthermore, each $\phi \in \mathrm{Hom}_{\sF}(P, Q)$ decomposes as an isomorphism followed by an inclusion.

 If $\text{Iso}_{\sF}(P, Q) \neq \varnothing$ we say that $P$ and $Q$ are \emph{$\sF$-conjugate}. 
\end{defn}
\begin{defn}\label{generated defn}
Let $\sF$ be a fusion system on $S$. We say \emph{$\sF$ is realised by $G$}, written as $\sF = \sF_S(G)$, if $S \leq G$ and for all $P, Q \leq S$
\[\mathrm{Hom}_{\sF}(P, Q) = \{\phi \colon P \rightarrow Q \colon \exists x \in G \text{ such that } \phi = c_x|_P\}.\]
If there exists a $G$ with $S \in \text{Syl}_p(G)$ such that $\sF$ is realised by $G$, then $\sF$ is \emph{non-exotic}, and exotic if such a $G$ does not exist and $\sF$ is saturated.
\end{defn}

\begin{lem}[Thm. 3 in \cite{realising}]\label{generating group existence}
For any fusion system $\sF$ on $S$, there exists some finite group $G$ such that $\sF$ is realised by $G$.
\end{lem}

For most purposes the above definition of a fusion system is too vague and it is common to impose the \emph{saturation conditions} and say that a fusion system satisfying them is \emph{saturated}. These conditions are designed so that saturated fusion systems are virtually indistinguishable from the following family of examples: 
\begin{example}[Thm. I.2.3 in \cite{AKO}]
Given a finite group $G$ and $S \in \text{Syl}_p(G)$, the fusion system $\sF_S(G)$ is saturated.
\end{example}
The precise details of these conditions are not important to us and we refer the curious reader to Chapter I.2 of \cite{AKO} for more information. Alperin's fusion theorem tells us that all morphisms in a saturated fusion system arise as compositions of restrictions of $\sF$-automorphisms of \emph{essential subgroups} of $S$, together with $S$ itself. (see Chapter I.3 of \cite{AKO}). If $\sF$ is saturated and has no essential subgroups, we say that $\sF$ is essential rank $0$.

\begin{defn}Given an $s \in S$, the \emph{$\sF$-conjugacy class of $s$} is the set 
\[s^{\sF} := \{s' \in S \colon \exists P, Q \text{ s.t. } \exists \phi \in \mathrm{Hom}_{\sF}(P, Q) \text{ with } \phi(s') = s\}.\]
We will write $\text{cl}(\sF)$ for the set of $\sF$-conjugacy classes. If $x$ satisfies $|C_S(x)| \geq |C_S(y)|$ for all $y \in x^{\sF}$ we say that $x$ is \emph{fully $\sF$-centralised} and write $\text{fccl}(\sF)$ for an arbitrary set of fully $\sF$-centralised $\sF$-conjugacy class representatives.
\end{defn}
\begin{lem}[Lem. I.1.2a in \cite{AKO}]\label{fccl sylow}
If $\sF = \sF_S(G)$ with $S \in \text{Syl}_p(G)$ then $x \in \text{fccl}(\sF) \iff C_S(x) \in \text{Syl}_p(C_G(x))$.
\end{lem}
\begin{proof}
We note that $C_X(x) = C_X(\langle x \rangle)$ for $X \in \{S, G\}$ and apply Lemma I.1.2a of \cite{AKO}.
\end{proof}

We will now begin expositing the required background on $\sF$-stable characters and the ring $R_{\C}(\sF)$.
\begin{defn} 
For a finite group $G$ and a field $k$, we define $R_k^+(G)$ to be the \emph{semiring of characters (of $k$-representations) of $G$}, and $R_{k}(G)$ to be the \emph{ring of (virtual) characters (of $k$-representations) of $G$}, which is the Grothendieck completion of $R_k^+(G)$. We will also write $\text{cf}(G)$ for the set of complex valued class functions on $G$.
\end{defn}

\begin{defn} 
Let $f \in \text{cf}(S)$. Then \emph{$f$ is $\sF$-stable} if for each $P \leq S$ and $\phi \in \mathrm{Hom}_{\sF}(P, S)$, $f|_P = f|_Q \circ \phi$. We write $\text{cf}(\sF)$ for the set of $\sF$-stable class functions on $S$.
\end{defn}

By Lemma 1.3 in \cite{unique factorisation}, these $\sF$-stable class functions are precisely the class functions $f$ such that $f(x) = f(y)$ for all $x \in y^{\sF}$.

Since characters of $S$ are class functions, this definition of $\sF$-stability may be applied. It is clear that $\sF$-stability is preserved under product and addition, so the non-virtual $\sF$-stable characters from a subsemiring of non-virtual characters of $S$.
Given a field $k$, we will denote the semiring of $\sF$-stable characters (of $k$-representations) as $R_k^+(\sF)$.
We call the Grothendieck completion of $R_k^+(\sF)$ the ring of (virtual) $\sF$-stable characters (of $k$-representations) and denote it as $R_k(\sF)$. By Proposition 3.11 of \cite{aycin thesis} $R_{k}(\sF)$ is also the subring of $\sF$-stable virtual characters of $S$.

The following elementary lemma will be useful in the subsequent sections and provides a good family of examples of $\sF$-stable characters:
\begin{lem}\label{restriction}
If $\sF = \sF_S(G)$ for some $G$, then for any $\chi \in R_{k}(G)$, $\mathrm{Res}^G_S(\chi) \in R_{k}(\sF)$.
\end{lem}
\begin{proof}
Clearly $\text{Res}^G_S(\chi) \in R_{k}(S)$, all that remains is to show that $\mathrm{Res}^G_S(\chi)$ is $\sF$-stable. Because $\sF$ is realised by $G$, $\mathrm{Hom}_{\sF}(P,Q) = \mathrm{Hom}_{G}(P,Q)$ so $s' \in s^{\sF} \iff s' \in s^G$ for all $s, s' \in S$. Since $\chi \in \text{cl}(G)$, we have that for any two $s, s' \in S$ with $s' \in s^G$ then $\mathrm{Res}^G_S(\chi)(s) = \mathrm{Res}^G_S(\chi)(s')$.
\end{proof}
We end by discussing the $\sF$-indecomposables, which act as the irreducible characters for fusion systems.
\begin{defn}\label{ind def}
If $\chi \in R_{\C}^+(\sF)$ cannot be written as a sum of two other elements in $R_{\C}^+(\sF)$ we say that $\chi$ is \emph{$\sF$-indecomposable}. We write the set of $\sF$-indecomposable characters as $\text{Ind}(\sF)$.
\end{defn}
Much like characters of groups, any $\sF$-stable character can be written as a $\Z$-linear combination of indecomposables, and any $\sF$-stable class function may be written as a $\C$-linear combination.
\begin{prop}[Lem. 2.1, Cor 2.2 of \cite{completion theorem}] \label{the rank theorem}
 We have $\langle \text{Ind}(\sF) \rangle_{\Z} = R_{\C}(\sF)$ and $\langle \text{Ind}(\sF) \rangle_{\C} = \text{cf}(\sF)$. Viewed as free abelian group, the rank of $R_{\C}(\sF)$  is $|\text{cl}(\sF)|$. 
\end{prop}

Unfortunately, in contrast to ordinary character theory of groups, each $\sF$-stable character need not factor \emph{uniquely} into a sum of $\sF$-indecomposables, requiring us to instead fix a $\Z$-basis $B$ of $R_{\C}(\sF)$ as opposed to assuming that $\text{Ind}(\sF)$ is linearly independent. We define the \emph{$\sF$-character table with respect to $B$} as the matrix $(X_B(\sF))_{\psi \in B, x \in \text{fccl}(\sF)} = \psi(x)$.

\begin{prop}\label{transitive prop}
If $\sF$ is transitive ($|\text{cl}(\sF)| = 2$), then Conjecture A holds.
\end{prop}
\begin{proof}
By Lemma 2.18 of \cite{unique factorisation} the $\sF$-indecomposables are the trivial character $1_S$ of $S$ and $\rho_S-1_S$, hence $X_{\text{Ind}(\sF)}(\sF)$ is
\[\begin{array}{|c||c|c|} \hline
\text{Ind}(\sF) & 1 & z \\ \hline \hline
1_S & 1 & 1 \\ 
\rho_S-1_S & |S|-1 & -1 \\ \hline
\end{array}\]
As $S$ is a $p$-group we have $|\mathcal{Z}(S)| > 1$ hence a fully $\sF$-centralised representative of the non-trivial $\sF$-conjugacy classes is central in $S$, yielding
\[\prod_{x \in \text{fccl}(\sF)} |C_S(x)| = |C_S(1)||C_S(z)| = S^2 = |\mathrm{det}(X_B(\sF))|^2.\]
\end{proof}

\fakesection{The representation ring of a product of fusion systems.}
In this section we will show that if $\sF$ is a minimal counterexample to Conjecture A, $\sF$ cannot be a product $\sF =\sF_1 \times \sF_2$ of two strictly smaller fusion systems.

\begin{defn}[Thm. I.6.6 in \cite{AKO}]\label{product fusion defn}
Let $\sF_i$ be fusion systems over $p$-groups $S_i$ for $i = 1,2$, then \emph{the product fusion system} $\sF_1 \times \sF_2$ is a fusion system over $S_1 \times S_2$ with
\[\mathrm{Hom}_{\sF_1 \times \sF_2}(P,Q) = \{(\phi_1, \phi_2)|_P \colon \phi_i \in \mathrm{Hom}_{\sF_i}(P_i, Q_i), (\phi_1, \phi_2)(P) \leq Q\}\]
Such that $P, Q \leq S_1 \times S_2$ and $P_i, Q_i$ denoting the projections of $P, Q$ to $S_i$.
\end{defn}

\begin{lem}\label{product classes}
If $(p_1, p_2) \in (q_1, q_2)^{\sF_1 \times \sF_2} \iff p_1 \in {q_1}^{\sF_1}$ and $p_2 \in {q_2}^{\sF_2}$.
\end{lem}
\begin{proof} By definition, if there is an $\sF_1\times\sF_2$-isomorphism sending $(p_1, p_2) \mapsto (q_1, q_2)$ we may assume it is of the form $(\phi_1, \phi_2) \in \text{Iso}_{\sF_1}(P_1, Q_1)\times\text{Iso}_{\sF_2}(P_2, Q_2)$ 
where $p_i \in P_i \leq S_i, q_i \in Q_i \leq S_i$. So $p_i \in {q_i}^{\sF_1}$ with $\phi_i$ being a map in $\sF_i$ such that $\phi_i(p_i) = q_i$.

If we now assume $p_i \in {q_i}^{\sF_i}$ with $\phi_i$ being the isomorphism mapping $p_i$ to $q_i$. Define $P_i, Q_i$ as before, 
then $(\phi_1, \phi_2) \in \text{Iso}_{\sF_1 \times \sF_2}((P_1, P_2), (Q_1, Q_2))$, so $(p_1, p_2) \in (q_1, q_2)^{\sF_1 \times \sF_2}$. 
\end{proof}

\begin{lem}\label{product centralised}
Let $\sF_i$ be a fusion system on $S_i$ for $i = 1,2$, then $(x_1,x_2)$ is fully $\sF=\sF_1 \times \sF_2$-centralised if and only if $x_i$ is fully $\sF_i$-centralised.
\end{lem}
\begin{proof}
It is clear that $|C_{S_1 \times S_2}((x_1, x_2))| = |C_{S_1}(x_1) \times C_{S_2}(x_2)|$. So if $x_i$ is fully $\sF_i$-centralised, then 
$|C_{S_1 \times S_2}((x_1, x_2))| =|C_{S_1}(x_1)||C_{S_2}(x_2)| \geq |C_{S_1}(x_1')||C_{S_2}(x_2')| = |C_{S_1 \times S_2}((x_1', x_2'))|$ 
for all $x_i' \in x_i^{\sF}$. Because $(x_1, x_2)^{\sF} = x_1^{\sF_1} \times x_2^{\sF_2}$ by Lemma \ref{product classes}, we have our result.
\end{proof}

We now describe $R_{\C}(\sF_1 \times \sF_2)$:
\begin{lem}\label{product is f-stable}
Let $\sF_1 \times \sF_2$ be a product fusion system on $S_1 \times S_2$ and $k$ a field, then $\chi_i \in R_{k}(\sF_i)$, then $\chi_1\chi_2 \in R_{k}(\sF_1 \times \sF_2)$.
\end{lem}
\begin{proof}
We first assume $\chi_i$ are $\sF_i$-stable. If $(p_1, p_2) \in (q_1, q_2)^{\sF_1\times\sF_2}$ then $p_1 \in q_1^{\sF_1}$ and $p_2 \in q_2^{\sF_2}$. 
Because  $\chi_1, \chi_2$ are $\sF_1$-stable and $\sF_2$-stable respectively, we have $\chi(p_1, p_2) = \chi_1(p_1)\chi_2(p_2) = \chi_1(q_1)\chi_2(q_2) = \chi(q_1, q_2)$. Thus $\chi$ is $\sF_1\times\sF_2$-stable. 
\end{proof}

\begin{thm}\label{min counter-example}
For any fusion systems $\sF_1, \sF_2$ over $p$-groups $S_1, S_2$ then, as commutative monoids and abelian groups respectively
\begin{align*}R_{\C}^+(\sF_1 \times \sF_2) &\cong R_{\C}^+(\sF_1) \otimes_{\mathbb{N}} R_{\C}^+(\sF_2) \\
R_{\C}(\sF_1 \times \sF_2) &\cong R_{\C}(\sF_1) \otimes_{\Z} R_{\C}(\sF_2).
\end{align*}
Furthermore, if Conjecture A holds for $\sF_1$ and $\sF_2$, it holds for $\sF_1 \times \sF_2$.
\end{thm}
\begin{proof}
By Lemma \ref{product classes} $|\text{cl}(\sF_1 \times \sF_2)| = |\text{cl}(\sF_1)| \times |\text{cl}(\sF_2)|$ and so the claimed isomorphisms exist by Theorem \ref{the rank theorem} and comparing ranks. As all representation rings involved are free, it follows from Proposition 1.10.1 of \cite{martinet} that the volume of $R_{\C}(\sF_1 \times \sF_2)$ as a $\Z$-lattice is $\text{Vol}(R_{\C}(\sF_1))^{|\text{cl}(\sF_2)|}\text{Vol}(R_{\C}(\sF_2))^{|\text{cl}(\sF_1)|}$. If $B$ is a basis for $\sF_1 \times \sF_2$ then by our assumption that  $\sF_1$ and $\sF_2$ satisfy Conjecture A we have
\begin{align*}|\mathrm{det}(X_{B}(\sF))|^2 &= \left(\prod_{x \in \text{fccl}(\sF_1)} |C_{S_1}(x)|\right)^{|\text{cl}(\sF_1)|}\left(\prod_{y \in \text{fccl}(\sF_2)}(|C_{S_2}(y)|\right)^{|\text{cl}(\sF_2)|} 
\\ &= \prod_{x \in \text{cl}(\sF_1)}\prod_{y \in \text{cl}(\sF_2)}|C_{S_1}(x)||C_{S_2}(y)| \\ &= \prod_{x \in \text{fccl}(\sF_1)}\prod_{y \in \text{fccl}(\sF_2)}|C_{S_1 \times S_2}(x, y))|.
\end{align*}
Applying Lemma \ref{product centralised} yields
\[\prod_{x \in \text{fccl}(\sF_1)}\prod_{y \in \text{fccl}(\sF_2)}|C_{S_1 \times S_2}((x, y))| = \prod_{z \in \text{fccl}(\sF_1 \times \sF_2)}|C_{S_1 \times S_2}(z)|\]
and the result is shown.
\end{proof}
\begin{cor}
If $\sF := \sF_1 \times \sF_2$ and $B_i$ is a basis for $R_{\C}(\sF_i)$, then $B :=  \{\psi\mu \colon \psi \in B_1, \mu \in B_2\}$ is a basis for $R_{\C}(\sF)$ and $X_{B}(\sF) = X_{B_1}(\sF_1) \otimes X_{B_2}(\sF_2)$, where we use $\otimes$ to denote the Kronecker product of matrices.
\end{cor}
\begin{proof}
Writing $S(M)$ for the set of eigenvalues of $M$ with multiplicity, then $S(A \otimes B) = \{ab \colon a \in S(A), b \in S(B)\}$, hence
\[\text{det}(A \otimes B) = \text{det}(A)^{\text{rk}(B)}\text{det}(B)^{\text{rk}(A)}.\]
It then follows that $\text{det}(X_B(\sF)) = \text{det}(X_{B_1}(\sF_1))^{|\text{cl}(\sF_2)|}\text{det}(X_{B_2}(\sF_2))^{|\text{cl}(\sF_1)|}$. It then follows Proposition 1.10.1 of \cite{martinet} that $B$ is a basis for $R_{\C}(\sF)$.
\end{proof}
\fakesection{An Inverted Brauer Character Theory}
In  unpublished notes \cite{olsson} Olsson generalised $p$-Brauer characters to $\pi$-Brauer characters for a \emph{set} of primes $\pi$ which are class functions on the conjugacy classes of $\pi'$-elements of $G$. We will investigate the connection between fusion systems over a $p$-group and $p'$-Brauer characters (the case where $\pi = \{q \text{ prime } \colon q \neq p\})$, providing a number of proofs for this specific case not given in \cite{olsson}. We refer the reader to Chapter 2 of \cite{navarro} for all of the required background on Brauer character theory.

\begin{lem}[Thm. 2.12 in \cite{navarro}, Thm. 43.ii in \cite{serre}] \label{brauer lifting}
Let $q$ be a prime and $G$ be a finite group. Then there is a ring homomorphism $L \colon R_{\overline{\F}_q}(G) \rightarrow R_{\C}(G)$ sending $\psi \in R_{\overline{\F}_q}(G)$ to the class function $L(\psi) \colon \text{cl}(G) \rightarrow \C$ such that $L(\psi)(g) = \psi(g_{q'})$, where $g_{q'}$ is the $q'$-part of $g \in G$. Moreover, if $(|G|, q) = 1$ then this map is an isomorphism sending $\text{IBr}_q(G) \rightarrow \text{Irr}(G)$.
\end{lem} 
We refer $L$ as the \emph{Brauer lifting} map and $L(\psi)$ as the \emph{Brauer lift} of $\psi$. The theory described in this section was originally inspired by the following observation:
\begin{prop}\label{two prime factors}
Let $|G| = p^aq^b$, $S \in \mathrm{Syl}_p(G)$ and $\sF = \sF_S(G)$. Write $\text{IBr}_q(G)$ for the set of irreducible $q$-Brauer characters of $G$. Then writing $\text{IBr}_q(G)|_S := \{\mathrm{Res}_S^G(\psi) \colon \psi \in \text{IBr}_q(G)\}$, $\text{Ind}(\sF) = L(\text{IBr}_q(G)|_S)$.
\end{prop}
\begin{proof}
Note that $\text{cl}_{q'}(G) = \text{cl}_{p}(G)$ and $\text{cl}(\sF) = \{K \cap S \colon K \in \text{cl}_{p}(G)\} $. Because $S$ is a $q'$-group, by Lemma \ref{brauer lifting} the Brauer lift $L \colon R_{\overline{\F}_q}(S) \rightarrow R_{\C}(S)$ is an isomorphism, thus $L(\mathrm{Res}_S^G(\psi)) = \mathrm{Res}_S^G(L(\psi)) \in R_{\C}(S)$ for $\psi \in \text{IBr}_q(G)$. Furthermore $\psi$ is invariant on $\text{cl}_{q'}(G) = \text{cl}_p(G)$ so we have $L(\mathrm{Res}_S^G(\psi)) \in R_{\C}(\sF)$.

Let $\chi \in \text{Ind}(\sF)$. We note that $\chi = \text{Res}^G_S(\psi)$ for some non-virtual $q$-Brauer character $\psi$ of $G$, and by Lemma \ref{brauer lifting} $L$ restricts to a bijection $\text{IBr}_q(S) \rightarrow \mathrm{Irr}(S)$ and that lifts of non-virtual $q$-Brauer characters of $S$ are non-virtual characters.

Suppose we can write $\chi = L(\mathrm{Res}^G_S(\psi_1+\psi_2)) = L(\mathrm{Res}_S^G(\psi_1))+L(\mathrm{Res}_S^G(\psi_2))$ for two $\psi_1, \psi_2 \in \text{IBr}_q(G)$. This immediately contradicts the $\sF$-indecomposability of $\chi$ as $L(\mathrm{Res}_S^G(\psi_i)) \in R^+_{\C}(\sF)$ for $i = 1,2$ by the previous paragraph. Thus $\chi = L(\mathrm{Res}_S^G(\psi))$ for some $\psi \in \text{IBr}_q(G)$ and by injectivity of $L$ we are done.
\end{proof}
This result motivates the  following definitions:
\begin{defn}\label{decomp matrix defn}
Let $\sF$ be a fusion system over $S$, $B$ a basis of $R_{\C}(\sF)$ and $G$ a finite group realising $\sF$. For each  $\chi \in \text{Irr}(G)$ we have a decomposition
\[\mathrm{Res}^G_S(\chi) = \sum_{\psi \in B} d_{\chi\psi}^{B,G}\psi\]
and we refer to the coefficients $d_{\chi\psi}^{B,G}$ as the \emph{decomposition numbers}. We collect these coefficients in the \emph{decomposition matrix of $\sF$ (with respect to $B$ and $G$)}: 
\[(D_{B,G}(\sF))_{\chi \in \mathrm{Irr}(G), \psi \in B} := d_{\chi\psi}.\]
We also define \emph{Cartan matrix of $\sF$ (with respect to $B$ and $G$)} to be
\[C_{B, G}(\sF) := D_{B, G}(\sF)^TD_{B, G}(\sF).\]
The Cartan matrix is indexed by $\psi, \mu \in B$. 
We define the \emph{Cartan numbers} $c^{B,G}_{\psi\mu} := C_{B,G}(\sF)_{\psi\mu}$.
\end{defn}
The following definition mirrors the definition of the projective indecomposable characters from Brauer character theory, which are the characters of the projective covers of the simple representation corresponding to an irreducible Brauer character (c.f. Section 18.3 of \cite{serre}).
\begin{defn}
For $\psi \in B$, we define the associated \emph{orthogonal indecomposable} to be \[\displaystyle\Phi^{B,G}_{\psi} := \sum_{\chi \in \mathrm{Irr}(G)} d^{B,G}_{\chi\psi}\chi.\] 
We set $(P_{B,G})_{\psi \in B, x^{\sF} \in \text{cl}(\sF)} := \Phi^{B,G}_{\psi}(x)$ and write $\mathrm{Orth}_{G, B}(\sF) := \{\Phi^{B,G}_{\psi} \colon \psi \in B\}$.
\end{defn}

For the purpose of investigating Conjecture A we show that we may fix an arbitrary choice for B:
\begin{lem}[Page 10 of \cite{geometry of numbers}] \label{basis invariant lemma} 
For any two $\Z$-bases $B,B'$ of a free $\Z$-module $V$ with associated matrices $X_B, X_{B'}$, then $|\mathrm{det}(X_B)| = |\mathrm{det}(X_B')|$.
\end{lem}
\begin{proof}
Given two bases $B, B'$ of $V$, there is a change of basis matrix $M \in GL_{|B|}(\Z)$ such that $MX_B = X_{B'}$, hence $\mathrm{det}(M)\mathrm{det}(X_B) = \pm\mathrm{det}(X_{B}) = \mathrm{det}(X_{B'})$ and therefore $|\mathrm{det}(X_B)|^2 = |\mathrm{det}(X_{B'})|^2$.
\end{proof}
We will now fix both $B$ and $G$ and omit them from our notation. We prove some elementary but critical properties of the orthogonal indecomposables.
\begin{lem}
If $\sF = \sF_S(G)$, then $\text{Orth}(\sF) \subset R_{\C}(G)$.
\end{lem}
\begin{proof}
Since $\mathrm{Res}_S^G(\chi) \in R_{\C}^+(\sF)$ for all $\chi \in \mathrm{Irr}(G)$ we have that the decomposition numbers are integers, hence $\Phi_{\psi} \in R_{\C}(G)$. 
\end{proof}
\begin{lem}[Lem. 6.2 in \cite{olsson}]
\label{orthogonality}
Let $X(\sF)$ be the character table of $\sF$ (with respect to $B$) and $\Delta :=  \text{Diag}_{x^{\sF} \in \text{cl}(\sF)}(|C_G(x)|)$. Then $\overline{P^T}X(\sF) = \Delta$ and $X(\sF)$ and $P$ are both of full rank.
\end{lem}
\begin{proof}
Let $g \in G$, $s \in S$, then:
\begin{align*}
\sum_{\psi \in B} \overline{\Phi_{\psi}(g)}\psi(s) &= \sum_{\psi \in B}\sum_{\chi \in \mathrm{Irr}(G)} d_{\chi\psi}\overline{\chi(g)}\psi(s) = \sum_{\chi \in \mathrm{Irr}(G)} \overline{\chi(g)}\sum_{\psi \in B}d_{\chi\psi}\overline{\psi(s)} = \sum_{\chi \in \mathrm{Irr}(G)} \overline{\chi(g)}\chi(s) \\ &= \delta_{g^{G}s^{G}}|C_G(g)|
\end{align*}
by column orthogonality (c.f. Theorem 16.4 of \cite{james liebeck}). Writing out $\overline{P^T}Q$ gives us:
\begin{align*}
(\overline{P^T}X(\sF))_{x^{\sF}, y^{\sF} \in \text{cl}(\sF)} &= \sum_{\psi \in B}\overline{\Phi_{\psi}(x)}\psi(y) = \delta_{x^{\sF},y^{\sF}}|C_G(x)| \\
&\Rightarrow \overline{P^T}X(\sF) = \text{Diag}_{x^{\sF} \in \text{cl}(\sF)}(|C_G(x)|) = \Delta.
\end{align*}
Since $|C_G(x)| \neq 0$ for any $x$, $\Delta$ is of full rank.  
We know that $P$ is square as $|B| = |\text{cl}(\sF)|$ by Proposition \ref{the rank theorem}, therefore 
\[|\text{cl}(\sF)| = \text{rk}(\Delta) \leq \text{min}(\text{rk}(P), \text{rk}(X(\sF))) \leq |\text{cl}(\sF)| \Rightarrow \text{rk}(P) = \text{rk}(X(\sF)) = |\text{cl}(\sF)|.\] 
and both results are proven.
\end{proof}

\begin{lem}\label{zcf}
$\Phi_{\psi}(g) = 0$ whenever $g$ is not a $p$-element of $G$.
\end{lem}
\begin{proof}
Let $g \in G$ such that $g$ is not a $p$-element. 
Then we have that $g^G \cap S = \varnothing$, hence $\sum_{\psi \in B} \overline{\Phi_{\psi}(g)}\psi(s) = 0$ for all $s \in S$ by Lemma \ref{orthogonality}. 
By the linear independence of $B$ we conclude $\Phi_{\psi}(g) = 0$ for all $\psi$. 
\end{proof}
\begin{cor}\label{zcf basis} Let zcf$_p(G)$ be the set of complex-valued class functions $f$ of $G$ with $f(g) = 0$ for all $g \in G$ that are not $p$-elements. Then $\{\Phi_{\psi}\}_{\psi \in B}$ is a $\C$-basis for zcf$_p(G)$. \end{cor}
\begin{proof}
By Lemma \ref{zcf} we know $\Phi_{\psi} \in \text{zcf}_p(G)$ for all $\psi \in B$, and by Lemma \ref{orthogonality} $P$ is full rank hence $\{\Phi_{\psi}\}_{\psi \in B}$ is linearly independent. 
Combining this with Proposition \ref{the rank theorem} we have $\text{rk}_{\C}(\text{zcf}_p(G)) = |\text{cl}(\sF)| = |B| = \text{rk}_{\C}(\langle \{\Phi_{\psi}\}_{\psi \in B} \rangle)$ and we are done.
\end{proof}

\begin{cor}\label{cor 3.6} $|G|_{p'}$ divides $\Phi_{\psi}(1)$. \end{cor}
\begin{proof}
Let $q \neq p$ be a prime dividing $|G|$,  $Q \in \mathrm{Syl}_q(G)$ and $1_Q$ the trivial character of $Q$. We have 
\[\langle \mathrm{Res}^G_Q(\Phi_{\psi}), 1_Q \rangle_Q = \frac{1}{|Q|}\sum_{q \in Q} \Phi_{\psi}(q),\]
which is $\frac{\Phi_{\psi}(1)}{|Q|}$ as $Q \cap S = \{1\}$ and $\Phi_{\psi}(q) = 0$ for all $q \not \in S$ by Corollary \ref{zcf}. 
Since $\Phi_{\psi} \in R_{\C}(G)$, $\langle \mathrm{Res}_Q^G(\Phi_{\psi}), 1_Q\rangle_Q \in \Z$, and so $|Q|$ divides $\Phi_{\psi}(1)$. This argument holds for every $q \neq p$ dividing $|G|$, hence $|G|_{p'}$ must divide $\Phi_{\psi}(1)$.
\end{proof}

\begin{prop}\label{oths inner product}
We have  $\langle \Phi_{\psi}, \Phi_{\mu} \rangle_G = c_{\psi\mu}$ for any $\psi, \mu \in B$.
\end{prop}
\begin{proof}
$\text{Irr}(G)$ is orthonormal under $\langle -,- \rangle_G$, hence by the definition of the Cartan numbers we have
\begin{align*}
 \langle \Phi_{\psi}, \Phi_{\mu} \rangle_G &= \sum_{\chi \in \mathrm{Irr}(G)}\sum_{\chi' \in \mathrm{Irr}(G)}d_{\psi\chi}d_{\mu\chi'}\langle \chi, \chi' \rangle = \sum_{\chi \in \mathrm{Irr}(G)} d_{\psi\chi}d_{\mu\chi} 
\\ &= (D(\sF)^TD(\sF))_{\psi\mu} = (C(\sF))_{\psi\mu} = c_{\psi\mu}.
\end{align*}
\end{proof}

The Cartan numbers in Brauer character theory describe how the projective indecomposables of $G$ decompose when restricted to the $p'$-elements of $G$ (see the remark on page 25 of \cite{navarro}). We have a similar result:

\begin{cor}\label{proj restriction} 
For all $\psi \in B$, $\displaystyle \mathrm{Res}^G_S(\Phi_{\psi}) = \sum_{\mu \in B} c_{\psi\mu}\mu$. 
\end{cor}
\begin{proof}
By definition of the Cartan numbers, $c_{\psi\mu} = \displaystyle\sum_{\chi \in \mathrm{Irr}(G)} d_{\chi\psi}d_{\chi\mu}$, thus
\[\text{Res}^G_S(\Phi_{\psi}) = \sum_{\chi \in \mathrm{Irr}(G)} d_{\chi\psi} \mathrm{Res}^G_S(\chi) = \sum_{\chi \in \mathrm{Irr}(G)}\sum_{\mu \in B} d_{\chi\psi}d_{\chi\mu}\mu = \sum_{\mu \in B} c_{\psi\mu}\mu.\]
\end{proof}

We will now apply this theory to completely describe the orthogonal indecomposables, decomposition and Cartan matrices of essential rank $0$ systems and saturated fusion systems over $D_8$.
\begin{example}
We label the characters of $D_8 = \langle r, t \colon r^2 = t^4 = 1, rtr^{-1} = t^{-1} \rangle$ as follows: \small
\[\begin{array}{|c||c|c|c|c|c|} \hline
& 1 & t^2 & r & rt & t \\ \hline \hline
\chi_1 & 1 & 1 & 1 & 1 & 1 \\ 
\chi_2 & 1 & 1 & 1 & -1 & -1 \\
\chi_3 & 1 & 1 & -1 & 1 & -1 \\
\chi_4 & 1 & 1 & -1 & -1 & 1 \\
\phi & 2 & -2 & 0 & 0 & 0 \\ \hline
\end{array}\]
\normalsize It is routine to calculate that \small
\[\begin{array}{|c||c|c|c|c|c|}\hline
\text{Ind}(\sF_{D_8}(S_4)) & 1 & t^2 & rt & t \\ \hline \hline
\chi_1 & 1 & 1 & 1 & 1 \\ 
\chi_2 & 1 & 1 & -1 & -1 \\
\chi_3+\phi & 3 & -1 & 1 & -1 \\
\chi_4+\phi & 3 & -1 & -1 & 1 \\\hline
\end{array}\quad
\begin{array}{|c||c|c|c|}\hline
\text{Ind}(\sF_{D_8}(A_6)) & 1 & t^2 & t \\ \hline\hline
\chi_1 & 1 & 1 & 1 \\ 
\chi_4+\phi & 3 & -1 & 1 \\ 
\chi_2+\chi_3+\phi & 4 & 0 & -2 \\ \hline
\end{array}\]
\normalsize and thus with respect to the above ordering of $\text{Ind}(\sF_{D_8}(S_4))$ and $\text{Ind}(\sF_{D_8}(A_6))$ we have the following decomposition matrices:
\[D(\sF_{D_8}(S_4)) = \small\begin{pmatrix} 1 & 0 & 0 & 0 \\ 0 & 1 & 0 & 0 \\ 1 & 1 & 0 & 0 \\ 0 & 0 & 1 & 0 \\ 0 & 0 & 0 & 1 \end{pmatrix} \quad \normalsize D(\sF_{D_8}(A_6)) = 
\small\begin{pmatrix} 1 & 1 & 1 & 1 & 1 & 2 & 0 \\ 0 & 0 & 0 & 1 & 1 & 1 & 2 \\ 0 & 1 & 1 & 1 & 1 & 1 & 1 \end{pmatrix}^T . 
\]\normalsize
Thus the Cartan matrices are
\[C(\sF_{D_8}(S_4)) = \small \begin{pmatrix} 2 & 1 & 0 & 0 \\ 1 & 2 & 0 & 0 \\ 0 & 0 &1 & 0\\ 0 & 0 & 0 & 1 \end{pmatrix} \normalsize \quad C(\sF_{D_8}(A_6)) = 
\small\begin{pmatrix} 9 & 4 & 6 \\ 4 & 7 & 5 \\ 6 & 5 & 6 \end{pmatrix} \normalsize
\]
with determinants $3 = |S_4|_{2'}$ and $45 = |A_6|_{2'}$ respectively. Hence the non-zero values of the orthogonal indecomposables are:
\small\[
\begin{array}{|c||c|c|c|c|}\hline
\text{Orth}(\sF_{D_8}(S_4)) & 1 & t^2 & rt & t \\   \hline\hline
\Phi_{\chi_1} & 3 & 3 & 1 & 1 \\
\Phi_{\chi_2} & 3 & 3 & -1 & -1 \\
\Phi_{\chi_3+\phi} & 3 & -1 & -1 & 1 \\
\Phi_{\chi_4+\phi}& 3 & -1 & 1 & -1 \\\hline
\end{array} \quad
\begin{array}{|c||c|c|c|} \hline
\text{Orth}(\sF_{D_8}(A_6)) & 1 & t^2 & t \\ \hline\hline
\Phi_{\chi_1} & 45 & 5 & 1 \\
\Phi_{\chi_4+\phi} & 45 & -3 & 1 \\
\Phi_{\chi_2+\chi_3+\phi} & 45 & 1 & -1 \\ \hline
\end{array}
\] \normalsize
\end{example}
We will now focus on essential rank $0$ fusion systems. The first step is to prove Theorem D:
\begin{prop}\label{regular character}
Let $\sF = \sF_S(G)$ with $S$ not necessarily Sylow in $G$, and $\rho_S, \rho_G$ be the regular characters of $S$ and $G$. Then for any basis $B$ of $R_{\C}(\sF)$ we have
\[\rho_S = \sum_{\psi \in B} \frac{\Phi_{\psi}(1)}{[G:S]}\psi, \quad \rho_G = \sum_{\psi \in B} \psi(1)\Phi_{\psi}.\]
Furthermore, $\Phi_{\psi}(1)/[G:S]$ is an integer.
\end{prop}
\begin{proof}
We have
\begin{align*}
\mathrm{Res}^G_S(\rho_G) &= \sum_{\chi \in \text{Irr}(G)} \chi(1)\mathrm{Res}^G_S(\chi) = \sum_{\chi \in \text{Irr}(G)}\left(\chi(1)\sum_{\psi \in B} d_{\chi\phi}\psi\right) \\ 
&= \sum_{\psi \in B} \left(\psi \sum_{\chi \in \text{Irr}(G)} d_{\chi\psi}\chi(1)\right) = \sum_{\psi \in B} \psi\Phi_{\psi}(1)
\end{align*}
and because $\mathrm{Res}^G_S(\rho_G) = [G:S]\rho_S$ we have $\rho_S = \displaystyle \sum_{\psi \in B} \frac{\Phi_{\psi}(1)}{[G:S]}\psi$. Conversely, we have
{\allowdisplaybreaks\begin{align*}
\rho_G &= \sum_{\chi \in \text{Irr}(G)} \chi(1)\chi = \sum_{\chi \in \text{Irr}(G)}\left(\chi\sum_{\psi \in B} d_{\chi\phi}\psi(1)\right) \\ 
&= \sum_{\psi \in B} \left(\psi(1) \sum_{\chi \in \text{Irr}(G)} d_{\chi\psi}\chi\right) = \sum_{\psi \in B} \psi(1)\Phi_{\psi}.
\end{align*}}
Concluding the proof of the first statement, and we now show that $\frac{\Phi_{\psi}(1)}{[G:S]} \in \Z$. $\rho_S \in R_{\C}(\sF)$ implies there exist a unique set $\{a_{\psi}\}_{\psi \in B} \subset \Z$ such that $\rho_S = \sum_{\psi \in B} a_{\psi}\psi$. Because $\Q$ is a flat $\Z$-module, $B$ is a $\Q$-basis for $\Q \otimes R_{\C}(\sF)$. However, this introduces a linear relation 
\[0 = \rho_S-\rho_S = \sum_{\psi \in B}\left(a_{\psi}-\frac{\Phi_{\psi}(1)}{[G:S]}\right)\psi\]
in $\Q \otimes R_{\C}(\sF)$, which is a contradiction unless $a_{\psi} = \frac{\Phi_{\psi}(1)}{[G : S]}$. Hence $\frac{\Phi_{\psi}(1)}{[G:S]} \in \Z$.
\end{proof}

\begin{prop-example}
For an essential rank $0$ system $\sF_S(S \rtimes W)$ we write $\psi_{\chi} \in \text{Ind}(\sF)$ for the $\sF$-indecomposable corresponding to the orbit sum of $\chi$ under the conjugation action of $W$. Then the Cartan matrix w.r.t $B = \text{Ind}(\sF)$ is diagonal with 
\[C = \mathrm{Diag}_{\chi \in \mathrm{Irr}(S)/\mathrm{Out}_{\sF}(S)}\left(\sum_{\nu \in \mathrm{Irr}(S \rtimes W)} \langle \mathrm{Res}^{S \rtimes W}_S(\nu), \chi \rangle_S^2\right)\] and
\[\Phi_{\psi} = 
\mathrm{Ind}_S^{S \rtimes W}(\psi_{\chi}).
\]
\end{prop-example}
\begin{proof}
For a character $\chi$ of a normal subgroup $N \trianglelefteq G$ we write $I_G(\chi)$ for the stabiliser of $\chi$ under the conjugation action of $G$. Let $\chi \in \mathrm{Irr}(S)$. We have $(S \rtimes W)/S \cong W \cong \mathrm{Out}_{\sF}(S)$. Labelling the second isomorphism $\phi$, by definition of a semidirect product \begin{align*}
\mathrm{Ind}_S^{S \rtimes W}(\chi)(s) &= \displaystyle\sum_{\alpha \in \mathrm{Out}_{\sF}(S)}\chi(\alpha(s)) 
\\ 
&= |\text{Stab}_{\mathrm{Out}_{\sF}(S)}(\chi)|\displaystyle\sum_{\mu \in \chi^{\mathrm{Out}_{\sF}(S)}}\mu(s)
\\
&= [I_{S \rtimes W}(\chi):S]\displaystyle\sum_{w \in \tau}\chi^{w}(s)
\end{align*} 
for some transversal $\tau$ of $I_{S \rtimes W}(\chi)$ in $S \rtimes W$. Furthermore, $\mathrm{Ind}_S^{S \rtimes W}(\chi)(g) = 0$ for all $g \not \in S$ as $S \trianglelefteq S \rtimes W$. Now let $\nu \in \mathrm{Irr}(S \rtimes W)$, by Clifford's theorem we have
\[\mathrm{Res}^{S \rtimes W}_S(\nu) = \langle \mathrm{Res}^{S \rtimes W}_S(\nu) , \chi\rangle_S\sum_{w \in \tau}^{} \chi^{w}.\]
For some $\chi \in \text{Irr}(S)$. However, again by the definition of a semidirect product, we have $\chi^{w_i} = \chi \circ \alpha_i$ for some $\alpha_i \in \mathrm{Out}_{\sF}(S)$. As $\sF$ is essential rank $0$, we know by Lemma 2.17 of \cite{unique factorisation} that $\text{Ind}(\sF)$ is a unqiue basis of $\sF$-indecomposables for $R_{\C}(\sF)$, is it clear that each $\sF$-indecomposable $\sF$ is an orbit sum of some $\chi_{\psi} \in \mathrm{Irr}(S)$ under the action of $\mathrm{Out}_{\sF}(S)$. 
Therefore, the decomposition matrix has exactly one non-zero entry in each row. Therefore the Cartan matrix is diagonal with 
\[c_{\psi\psi} = \displaystyle\sum_{\nu \in \mathrm{Irr}(S \rtimes W)} d_{\nu\psi}^2 = \sum_{\nu \in \mathrm{Irr}(S \rtimes W)} \langle \mathrm{Res}_S^{S \rtimes W}(\nu), \chi_{\psi}\rangle^2.\] 
We will now determine $\mathrm{Orth}(\sF)$. By Proposition \ref{regular character} we have 
\[\sum_{\chi_{\psi} \in \text{Irr}(S)/W} \left(\frac{\Phi_{\psi}(1)}{|W|}\sum_{w \in \tau} \chi_{\phi}^w\right) = \sum_{\chi_{\psi} \in \text{Irr}(S)/W} \left(\chi(1)\sum_{w \in \tau} \chi_{\phi}^w\right).\]
Together with linear independence of irreducible characters we have $\Phi_{\psi}(1) = |W|\chi_{\psi}(1) = \text{Ind}_S^{S \rtimes W}(\chi_{\psi})(1)$. Next, for $\nu \in \text{Irr}(S \rtimes W)$ and $\mu \in \text{Irr}(S)$ such that $\langle \mathrm{Res}^{S \rtimes W}_S(\nu), \mu \rangle > 0$ we note that by Corollary \ref{proj restriction} and Frobenius reciprocity respectively, we have
\begin{align*}
\langle \Phi_{\psi}, \nu \rangle \neq 0 \iff \chi_{\psi} \in \mu^{\mathrm{Out}_{\sF}(S)} \\
\langle \text{Ind}_S^{S \rtimes W}(\chi_{\psi}), \nu \rangle \neq 0 \iff \chi_{\psi} \in \mu^{\mathrm{Out}_{\sF}(S)}.
\end{align*}
Finally, the Cartan matrix being diagonal together with Proposition \ref{oths inner product} implies that $\mathrm{Orth}(\sF)$ is orthogonal and no two associated orthogonal characters have any $S \rtimes W$-indecomposable constituents in common. Hence by Corollary \ref{zcf basis},  $\text{Ind}_S^{S \rtimes W}(\chi_{\psi})$ lies in $\langle \Phi_{\psi} \rangle_{\Q}$, but because they have equal degrees we have $\Phi_{\psi} = \text{Ind}_S^{S \rtimes W}(\chi_{\psi})$ and we are done.
\end{proof}
We conclude this subsection with two remarks. 
\begin{remark}
It is very common that each $S \rtimes W$-irreducible $\nu$ satisfies $\langle \mathrm{Res}^{S \rtimes W}_S(\nu), \chi \rangle = 0$ or $1$. When this occurs the Cartan numbers are $c_{\psi\psi} = |\{\nu \in \text{Irr}(S \rtimes W) \colon \mathrm{Res}_S^{S \rtimes W}(\nu) = \psi\}|$. When $S$ is abelian this reduces further to $c_{\psi\psi} = [S \rtimes W : I_{S \rtimes W}(\psi)]$ by the method of little groups (c.f. Section 8.2 of \cite{serre}) and the fact that $W$ is a $p'$-group.
\end{remark}\begin{remark}
Secondly, in many observed cases when $\sF = \sF_S(G)$ and $S \in \text{Syl}_p(G)$ we have $\Phi_{\psi} = \mathrm{Ind}_S^G(\psi)$, but this is not always so: Let $B$ be the basis of $R_{\C}(\sF_{C_3 \wr C_3}(A_9))$ described below: \footnotesize
\[\begin{array}{|c||c|c|c|c|c|c|c|} \hline
\sF_{C_3 \wr C_3}(A_9) & 1 & 3 & 3^2 & 3^3 & 9_1 & 9_2 \\\hline\hline
1 & 1 & 1 & 1 & 1 & 1 & 1 \\ \hline
\psi_1 & 8 & -4 & 2 & -1 & -1 & 2 \\ \hline
\psi_2 & 8 & -4 & 2 & -1 & 2 & -1 \\ \hline
\psi_3 & 8 & 5 & 2 & -1 & -1 & -1 \\ \hline
\psi_4 & 18 & 3 & -3 & 0 & 0 & 0 \\ \hline
\psi_5 & 20 & -4 & -1 & 2 & -1 & -1 \\ \hline
\end{array} 
\] \normalsize
Then $\Phi_{\psi_4} \in \mathrm{Orth}_B(\sF_{C_3 \wr C_3}(A_9))$ does not appear as $\mathrm{Ind}_{C_3 \wr C_3}^{A_9}(\chi)$ for any $\chi \in \text{Irr}(C_3 \wr C_3)$. This is the only example where $\Phi_{\psi} \neq \text{Ind}_S^G(\psi)$ for any $\psi$ in a given basis known to the author.
\end{remark}
\fakesection{The proof of Theorems A, B, C}
We will now begin applying our theory to first relate the determinant of the character table $X(\sF)$ to the determinant of the Cartan matrix $C(\sF)$ and then prove that the latter is coprime to $p$, allowing us to reduce the full fusion volume conjecture down to verifying $|\mathrm{det}(X(\sF))^2|$ is a power of $p$.

\begin{lem}[Thm. 6.9 in \cite{olsson}]\label{det fraction} 
For any fusion system $\sF$ over $S$ with $\sF = \sF_S(G)$ we have that 
\[\displaystyle|\mathrm{det}(X(\sF))|^2 = \frac{\prod_{x^{\sF} \in \text{cl}(\sF)} |C_G(x)|}{\mathrm{det}(C(\sF))}.\]
\end{lem}
\begin{proof}
Let $\chi \in \mathrm{Irr}(G)$ and $K \in \text{cl}(\sF)$, then:
\[(D(\sF)X(\sF))_{\chi, x^{\sF}} = \sum_{\mu \in B} D(\sF)_{\chi\mu}X(\sF)_{\mu, x^{\sF}} = \sum_{\mu \in B} d_{\chi \mu}\mu(x) = \text{Res}_S^G(\chi)(x).\]
So writing $x_1, ..., x_{|\text{cl}(\sF)|}$ for a set of $\sF$-conjugacy class representatives, by Lemma \ref{orthogonality}, we have that $(\overline{(D(\sF)X(\sF))^T}D(\sF)X(\sF))_{ij} = \delta_{ij}|C_G(x_{j})|$. Therefore \[\mathrm{det}(\overline{(D(\sF)X(\sF))^T}D(\sF)X(\sF)) = \prod_{i=1}^{|\text{cl}(\sF)|}|C_G(x_i)|.\] Finally 
\begin{align*}
\mathrm{det}(\overline{(D(\sF)X(\sF))^T}D(\sF)X(\sF)) &= \mathrm{det}(\overline{D(\sF)^T}D(\sF))\mathrm{det}(\overline{X(\sF)^T}X(\sF)) \\ 
&= \mathrm{det}(D(\sF)^TD(\sF))|\mathrm{det}(X(\sF))|^2 \\&=\mathrm{det}(C(\sF))|\mathrm{det}(X(\sF))|^2 \\ 
&\Rightarrow |\mathrm{det}(X(\sF))|^2 =  \frac{\prod_{i=1}^{|\text{cl}(\sF)|}|C_G(x_i)|}{\mathrm{det}(C(\sF))}.
\end{align*}
\end{proof}

\begin{cor}\label{det 2}
$|\mathrm{det}(X(\sF))|^2$ is an integer.
\end{cor}
\begin{proof}
Since $X(\sF)$ is a matrix of complex character values, $X(\sF) \in M_{|\text{cl}(\sF)|}(\mathcal{R})$ where $\mathcal{R} \subset \C$ is the ring of algebraic integers. 
So by Lemma \ref{det fraction} and the fact that the Cartan numbers are integers we have $|\mathrm{det}(X(\sF))|^2 \in \Q \cap \mathcal{R} = \Z$.
\end{proof}
We will now prove that $(\mathrm{det}(C(\sF)), p) = 1$. We note that by Corollary 2.18 in \cite{navarro}, the Cartan matrix of $p$-Brauer characters has determinant equal to a power of $p$.
\begin{lem}\label{cartan inverse} 
$C(\sF)^{-1}$ is the matrix defined by $\displaystyle(C(\sF)^{-1})_{\psi, \mu \in B} := \sum_{x^{\sF} \in \text{cl}(\sF)} \frac{\psi(x)\overline{\mu(x)}}{|C_G(x)|}$.
\end{lem}
\begin{proof}
Write $C'$ for the matrix $\displaystyle(C')_{\psi, \mu \in B} := \sum_{x^{\sF} \in \text{cl}(\sF)} \frac{\psi(x)\overline{\mu(x)}}{|C_G(x)|}$. Then
{\allowdisplaybreaks\begin{align*}
(C'C(\sF))_{\mu\psi} &= \sum_{\theta \in B} C'_{\mu\theta}C(\sF)_{\theta\psi}  
\\ &= \sum_{\theta \in B}\sum_{x^{\sF} \in\text{cl}(\sF)} \frac{\mu(x)\overline{\theta(x)}}{|C_G(x)|}c_{\theta\psi}
\\ &= \sum_{x^{\sF} \in \text{cl}(\sF)}\frac{\mu(x)\sum_{\theta \in B} c_{\theta\psi}\overline{\theta(x)}}{|C_G(x)|} 
\\ &= \sum_{x^{\sF} \in \text{cl}(\sF)} \frac{\mu(x)\overline{\mathrm{Res}_S^G(\Phi_{\psi})(x)}}{|C_G(x)|} \quad \text{ (applying Corollary \ref{proj restriction}).} 
\end{align*}}
Writing $\Delta :=  \text{Diag}_{x^{\sF} \in \text{cl}(\sF)}(|C_G(x)|)$, then we know from Lemma \ref{orthogonality} that
\[X(\sF)\overline{P^T} = \Delta \Rightarrow (X(\sF)\Delta^{-1}\overline{P^T})_{\mu, \psi} = \sum_{x^{\sF} \in \text{cl}(\sF)} \frac{\mu(x)\overline{\Phi_{\psi}(x)}}{|C_G(x)|} = \delta_{\mu,\psi}.\]
Since $x \in S$ we may replace $\Phi_{\psi}$ in the above expression with $\mathrm{Res}^G_S(\Phi_{\psi})$, hence $C'C = I_{|\text{cl}(\sF)|}$ and the result is proven.
\end{proof}

\begin{lem}\label{integer-valued matrix}
Assume $\sF = \sF_S(G)$ and $S \in \mathrm{Syl}_p(G)$, then $|G|_{p'}^2C(\sF)^{-1}$ is an integer-valued matrix.
\end{lem}
\begin{proof}
For $\chi \in R_{\C}(S)$ we define $\tilde{\chi}$ to be the class function
\[\tilde{\chi}(g) := \begin{cases} |G|_{p'}\chi(g) & g \text{ is a $p$-element} \\ 0 & \text{otherwise} \end{cases}\]
then $\tilde{\chi} \in R_{\C}(G)$ by Brauer's characterisation of characters (c.f. Thm. 2.1 of \cite{serre}). Hence for any $\mu, \psi \in B$, we have $\langle \tilde{\mu}, \tilde{\psi} \rangle_{G} \in \Z$.
Now since $\tilde{\mu}, \tilde{\psi}$ are $0$ on anything that isn't a $p$-element:
\[\langle \tilde{\mu}, \tilde{\psi} \rangle_G = \sum_{x^G \in \text{cl}(G)} \frac{\tilde{\mu}(x)\overline{\tilde{\psi}(x)}}{|C_G(x)|} =\sum_{\substack{x^G \in \text{cl}(G) \\ x^G \cap S \neq \varnothing}}\frac{\tilde{\mu}(x)\overline{\tilde{\psi}(x)}}{|C_G(x)|}.\]
All Sylow $p$-subgroups are conjugate in $G$, so we are free to choose our $G$-conjugacy class representatives $x$ to lie in $S$:
\begin{align*} 
\sum_{\substack{x^G \in \text{cl}(G) \\ x^G \cap S \neq \varnothing}}\frac{\tilde{\mu}(x)\overline{\tilde{\psi}(x)}}{|C_G(x)|} &= \sum_{x^{\sF} \in \text{cl}(\sF)}\frac{\tilde{\mu}(x)\overline{\tilde{\psi}(x)}}{|C_G(x)|} 
 \\ &= |G|_{p'}^2 \sum_{x^{\sF} \in \text{cl}(\sF)}\frac{\mu(x)\overline{\psi(x)}}{|C_G(x)|} 
 \\ &= |G|_{p'}^2(C(\sF)^{-1})_{\mu\psi} \in \Z
\end{align*}
concluding the proof.
\end{proof}

\begin{lem}[Thm 6.7 in \cite{olsson}] \label{cartan coprime}
If $\sF = \sF_S(G)$ with $S \in \mathrm{Syl}_p(G)$, then $\mathrm{det}(C)$ is coprime to $p$.
\end{lem}
\begin{proof}
Using the notation from the proof of Lemma \ref{integer-valued matrix}, define $M$ with $(M)_{\mu, \psi \in B} = \langle \tilde{\mu}, \tilde{\psi} \rangle_G$, 
which is $|G|^2_{p'}C(\sF)^{-1}$ by Lemma \ref{integer-valued matrix}. 
So $C(\sF)M = |G|^2_{p'}I_{|\text{cl}(\sF)|}$. Since $M$ and $C(\sF)$ are integral matrices, $\mathrm{det}(M), \mathrm{det}(C(\sF)) \in \Z$. 
Therefore 
\[\mathrm{det}(C(\sF))= \frac{|G|_{p'}^{2|\text{cl}(\sF)|}}{\mathrm{det}(M)} \in \Z \Longrightarrow \mathrm{det}(C(\sF)) \text{ is coprime to } p.\]
\end{proof}
Following this result, Lemma \ref{fccl sylow}, and Corollary \ref{det 2}, we have that
\[|\mathrm{det}(X(\sF))|^2_p = \frac{\prod_{x^{\sF} \in \text{cl}(\sF)} |C_G(x)|_p}{\mathrm{det}(C(\sF))_p} = \prod_{x \in \text{fccl}(\sF)} |C_S(x)|.\]
The final step is to show $|\mathrm{det}(X(\sF))|^2_p = |\mathrm{det}(X(\sF))|^2 \iff |\mathrm{det}(X(\sF))|^2$ is a power of $p$. 

\begin{defn}
Let $\rho \colon G \rightarrow GL_n(\overline{\F_p})$ be a $p$-modular representation of $G$. Let $G_{p'}$ be the set of $p'$-elements of $G$. Then the \emph{$p$-Brauer character} $\chi$ associated to $\rho$ is the map $G_{p'} \rightarrow \overline{\F_p}^{\times}$ given by $\chi(g) = \text{tr}(\rho(g))$.

For a given prime $\ell$, we write $R^+_{\overline{\F_{\ell}}}(\sF)$ to denote the subsemiring of $R^+_{\overline{\F_{\ell}}}(S)$ consisting of $\sF$-stable $\ell$-Brauer characters. As before, we write $R_{\overline{\F}_{\ell}}(\sF)$ for the Grothendieck completion of $R^+_{\overline{\F}_{\ell}}(\sF)$. We will write cf$_{p'}(G)$ for the set of functions $\text{cl}_{p'}(G) \rightarrow \overline{\F}_p$. 
\end{defn}
We will first take care of the easy case $\ell = p$.
\begin{prop}\label{mod p prop}
If $\sF$ is a fusion system over a $p$-group $S$, then $R_{\overline{\F_{p}}}(\sF) = \langle 1_S \rangle$ where $1_S$ is the trivial $p$-Brauer character. 
\end{prop}
\begin{proof}
By Corollary 2.10 in \cite{navarro}, we have that $|\text{IBr}_p(G)| = |\text{cl}_{p'}(G)|$. Because $S$ is a $p$-group $|\text{IBr}_p(S)| = |\text{cl}_{p'}(S)| = 1$, and so we have $\text{IBr}_q(S) = \{1_S\}$. $1_S$ is constant thus $\sF$-stable and we are done.
\end{proof}
Our strategy is to show that $X(\sF)$ is still full rank, thus invertible, ``mod $\ell$" for all $\ell \neq p$. We begin by making this reduction mod $\ell$ rigorous.
\begin{defn}\label{mod p defn}
For a given prime $\ell$, let $\mathcal{M}_{\ell}$ be a maximal ideal of the algebraic integers $\mathcal{R}$ with $\ell\Z \subseteq \mathcal{M}_{\ell}$. 
Denote the canonical surjection $\mathcal{R} \rightarrow \mathcal{R}/\mathcal{M}_{\ell} \cong \overline{\F_{\ell}}$ by $\pi_{\ell}$. 
Given a function $f$ that maps into $\mathcal{R}$, we abuse notation slightly and write $\pi_{\ell}(f)$ for the composition $\pi_{\ell} \circ f$.
\end{defn}

\begin{lem}\label{indicator functions in span}
Let $G$ be a finite group. For $K \in \text{cl}_{\ell'}(G)$, let $i_K \in cf(G)$ be the indicator function for $K$. 
Then $\pi_{\ell}(i_K)|_S \in \langle \text{IBr}_{\ell}(G)|_S \rangle_{\overline{\F_{\ell}}}$.
\end{lem}
\begin{proof}
By Lemma 2.4 and Theorem 1.19 in \cite{navarro} we have that $\text{IBr}_q(G)$ is linearly independent over $\overline{\F_{\ell}}$. Thus $\text{rk}_{\overline{\F_{\ell}}}(\langle \text{IBr}_{\ell}(G) \rangle_{\overline{\F_{\ell}}}) = |\text{cl}_{\ell'}(G)| = \text{rk}_{\overline{\F_{\ell}}}(\text{cf}_{\ell'}(G)) \Rightarrow \text{rk}_{\overline{\F_{\ell}}}(\langle \text{IBr}_{\ell}(G) \rangle_{\overline{\F_{\ell}}}) = \text{cl}_{\ell'}(G)$, the result then follows upon restricting to $S$.
\end{proof}
\begin{lem}[Lem. 5.2.2 in \cite{benson}]\label{benson lemma}
Let $\{f_i \colon R_1 \rightarrow R_2\}_{i=1}^n$ be a set of ring homomorphisms with $R_1$ a commutative ring and $R_2$ an integral domain. If the $f_i$ are distinct, they are linearly independent over $R_2$.
\end{lem}

\begin{thm}\label{mod rank theorem}
Let $\sF$ be any fusion system over a $p$-group $S$. For $\ell \neq p$, $\text{rk}_{\F_{\ell}}(R_{\overline{\F}_{\ell}}(\sF)) = |\text{cl}(\sF)|$.
\end{thm}
\begin{proof}
For any $s \in S$ we define the evaluation map $e_s \colon R_{\overline{\F}_{\ell}}(\sF) \otimes \overline{\F_{\ell}} \rightarrow \overline{\F_{\ell}}$ with $e_s(\psi) = \psi(s)$. 
It is clear that these maps are ring homomorphisms. 

Take a set of $\sF$-conjugacy class representatives $x_K$, with $x_K \in K \in \text{cl}(\sF)$, and consider the set $\mathscr{E} := \{e_{x_K} \colon K \in \text{cl}(\sF)\}$. 
We aim to show that each $e_{x_K}$ is distinct, so $\mathscr{E}$ is linearly independent over $\overline{\F_{\ell}}$ by Lemma \ref{benson lemma}.

By Corollary 2.10 in \cite{navarro}, the $\ell$-Brauer characters of $S$ form an $\overline{\F_{\ell}}$-basis for the class functions $S \rightarrow \overline{\F_{\ell}}$. Hence the indicator functions $\pi_{\ell}(i_K)$ are in $\overline{\F_{\ell}} \otimes_{\F_{\ell}} R_{\overline{\F_{\ell}}}(S)$, hence $\pi_{\ell}(i_K) \in \overline{\F_{\ell}} \otimes_{\F_{\ell}} R_{\overline{\F_{\ell}}}(\sF)$. We abuse notation and will also write $e_{x_K}$ for the extension of $e_{x_K}$ to $\overline{\F_{\ell}} \otimes_{\F_{\ell}} R_{\overline{\F_{\ell}}}(\sF)$.

Because $\pi_{\ell}(i_{K'})$ are indicator functions, we have $e_{x_K}(\pi_{\ell}(i_{K'})) = \delta_{K,K'}$ for any $K, K' \in \text{cl}(\sF)$ hence each $e_{x_K}$ is distinct and $\mathscr{E}$ is linearly independent over $\overline{\F_{\ell}}$ by Lemma \ref{benson lemma}.
Since $R_{\overline{\F_{\ell}}}(S)$ is a finite dimensional $\F_{\ell}$-algebra, $R_{\overline{\F_{\ell}}}(\sF) \otimes \overline{\F_{\ell}}$ is a finite dimensional $\overline{\F_{\ell}}$-algebra, yielding
\[\text{rk}_{\overline{\F_{\ell}}}(\langle\mathscr{E}\rangle) = |\text{cl}(\sF)| \leq \text{rk}_{\overline{\F_{\ell}}}(\mathrm{Hom}(R_{\overline{\F_{\ell}}}(\sF) \otimes \overline{\F_{\ell}}, \overline{\F_{\ell}})) = \text{rk}_{\overline{\F_{\ell}}}(R_{\overline{\F_{\ell}}}(\sF) \otimes \overline{\F_{\ell}}) = \text{rk}_{\F_{\ell}}(R_{\overline{\F_{\ell}}}(\sF)).\]
So $\text{rk}_{\F_{\ell}}(R_{\overline{\F}_{\ell}}(\sF)) \geq |\text{cl}(\sF)|$. 
We also have $R_{\overline{\F}_{\ell}}(\sF) \leq \langle \pi_{\ell}(i_k) \colon K \in \text{cl}(\sF) \rangle$, which is rank $|\text{cl}(\sF)|$.
 Therefore, $\text{rk}_{\F_{\ell}}(R_{\overline{\F_{\ell}}}(\sF)) = |\text{cl}(\sF)|$ as desired.
\end{proof}

\begin{cor}\label{b lin indep}
$\pi_{\ell}$ induces a bijection $B \mapsto \pi_{\ell}(B)$ and $\pi_{\ell}(B)$ is linearly independent over $\overline{\F_{\ell}}$.
\end{cor}
\begin{proof}
Take $\chi \in \text{IBr}_{\ell}(G)$, we may Brauer lift $\chi$ to a $L(\chi) \in R_{\C}(G)$. 
Then because $\mathrm{Res}^G_S(L(\chi))$ is $\sF$-stable we may write $\mathrm{Res}^G_S(L(\chi)) = \sum_{\psi \in B} c_{\psi}\psi$ for some $c_{\psi} \in \Z$.
Because each $c_{\psi} \in \Z \subset \mathcal{R}$, $\pi_{\ell}(c_{\psi})$ is well-defined. 
By Theorem 43 in \cite{serre} we have $\pi_{\ell}(\chi') = \chi$. So $\pi_{\ell}(\mathrm{Res}^G_S(L(\chi))) = \mathrm{Res}^G_S(\chi) = \sum_{\psi \in B} \pi_{\ell}(c_{\psi})\pi_{\ell}(\psi)$. 

By Lemma \ref{indicator functions in span} we have that $\langle \text{IBr}_{\ell}(G)|_S \rangle_{\overline{\F_{\ell}}} = R_{\overline{\F_{\ell}}}(\sF) \otimes \overline{\F_{\ell}}$, in particular:
\[R_{\overline{\F_{\ell}}}(\sF) \otimes \overline{\F_{\ell}} = \langle \text{IBr}_{\ell}(G)|_S \rangle_{\overline{\F_{\ell}}} = \langle \pi_{\ell}(B) \rangle_{\overline{\F_{\ell}}}.\]
Thus $\pi_{\ell}(B)$ spans $R_{\overline{\F_{\ell}}}(\sF) \otimes \overline{\F}_{\ell}$ as a $\overline{\F}_{\ell}$-vector space, 
which has rank $|\text{cl}(\sF)| = |B|$, so $|B| = |\pi_{\ell}(B)|$ and $\pi_{\ell}(B)$ is linearly independent.
\end{proof}
\begin{thm}\label{conj a for non-exotic}
Conjecture A holds for non-exotic saturated fusion systems.
\end{thm}
\begin{proof}
By Corollary \ref{det 2} we know that $|\mathrm{det}(X(\sF))|^2$ is an integer. 
Viewing an element of $B$ as a tuple in $\C^{|\text{cl}(\sF)|}$, then the rows of $X(\sF)$ are elements of $B$, 
so Corollary \ref{b lin indep} implies that $\pi_{\ell}(\mathrm{det}(X(\sF))) \neq 0$ for all $\ell \neq p$. 
Thus $|\mathrm{det}(X(\sF))|^2 \in \bigcap_{\ell \neq p} \Z-\ell\Z \Rightarrow |\mathrm{det}(X(\sF))|^2$ is a power of $p$. 
So by Lemma \ref{det fraction} and Lemma \ref{cartan coprime} we have that $|\mathrm{det}(X(\sF))|^2 = \prod_{x^{\sF} \in \text{cl}(\sF)} |C_G(x)|_{p}$ and $\mathrm{det}(C) = \prod_{x^{\sF} \in \text{cl}(\sF)} |C_G(x)|_{p'} = \prod_{x \in \text{fccl}(\sF)} |C_S(x)|$. 
\end{proof}

The obstructions to using Lemma \ref{generating group existence} to extend this proof method to all fusion systems are

\begin{obs}\label{obs1} If $S$ is not Sylow in $G$ then we do not have $C_S(x) \in \text{Syl}_p(C_G(x))$ for $x$ fully $\sF$-centralised, so $\prod_{x^{\sF} \in \text{cl}(\sF)} |C_G(x)|_p \neq \prod_{x \in \text{fccl}(\sF)} |C_S(x)|$. \end{obs}
\begin{obs}\label{obs2}
Lemma \ref{cartan coprime} requires $S$ to be Sylow in $G$ due to a dependence on Lemma \ref{integer-valued matrix}, where we utilise the fact that all Sylow subgroups are conjugate in order to choose $G$-conjugacy class representatives that lie in $S$.
\end{obs}

Despite these issues, we are still able to prove Theorem B:
\begin{thm}\label{exotic p power det}
Let $\sF$ be any fusion system over $S$, then $|\mathrm{det}(X(\sF))|^2$ is a power of $p$.
\end{thm}
\begin{proof}
$|\mathrm{det}(X(\sF))|^2 \in \Z$ and $\pi_{\ell}(\mathrm{det}(X(\sF))) \neq 0$ for all primes $\ell \neq p$ by Corollary \ref{det 2} and Corollary \ref{b lin indep} respectively. Hence $|\mathrm{det}(X(\sF))|^2 = p^{\alpha}$ for some $\alpha \in \mathbb{N}$ as before.
\end{proof}
We conclude by observing that the Obstruction \ref{obs1} increases the $p$-part of $|C_G(x)|$ whilst Obstruction \ref{obs2} increases the $p$-part of the determinant of the Cartan matrix. These two factors may cancel with one another. This idea leads to the following proposition:
\begin{prop}
Let $\sF$ be an exotic fusion system on $S$ and $G$ be some finite group with $\sF = \sF_S(G)$ as in Lemma \ref{generating group existence}. Then the fusion volume conjecture holds for $\sF$ if 
\[|\mathrm{det}(C(\sF))|_p = \frac{\prod_{x^{\sF} \in \text{cl}(\sF)}|C_G(x)|_p}{\prod_{x \in \text{fccl}(\sF)} |C_S(x)|}.\]
\end{prop}
\begin{proof}
We have that $|\text{det}(X(\sF))|^2_p = |\text{det}(X(\sF))|^2$ by Theorem \ref{exotic p power det} and thus under our assumption we have
\[|\text{det}(X(\sF))|^2 = \frac{\prod_{x^{\sF} \in \text{cl}(\sF)} |C_G(x)|_p}{\text{det}(C(\sF))_p} = \prod_{x \in \text{fccl}(\sF)} |C_S(x)|.\]
\end{proof}
\fakesectionnonumber{References}
\footnotesize
\renewcommand{\section}[2]{}

\vspace*{0.5cm}
\textsc{Department of Mathematics, Loughborough University, LE11 3TU,  United Kingdom} 
\newline
\textit{Email:} t.lawrence@lboro.ac.uk

\begin{thebibliography}{0}
\bibitem{AKO} M. Aschbacher, R. Kessar, B. Oliver \emph{Fusion Systems in Algebra and Topology} (2011), Cambridge University Press, 1st ed. ISBN 978-1-107-60100-0
\bibitem{james liebeck} G. James, M. Liebeck \emph{Representations and Characters of Groups} (2001) Cambridge Uni. Press,  2nd ed. ISBN 978-0-5118-1453-2
\bibitem{completion theorem}N. B\'arcenas, J. Cantaerro \emph{A Completion Theorem for Fusion Systems} (2020) Israel J. Math. \textbf{236}, pp. 501-531  doi:10.1007/s11856-020-1981-4
\bibitem{benson}D.J. Benson \emph{Representations and Cohomology I: Basic Representation Theory of Finite Groups and Associative Algebras} (1991) Cambridge University Press ISBN 978-0-511-62361-5 
\bibitem{unique factorisation} J. Cantarero, G. Combariza \emph{Uniqueness of Factorization for Fusion-Invariant Representations} (2023)  doi:10.48550/arXiv.2303.10341 
\bibitem{symmetric groups} J. Cantarero, J. Gaspar-Lara \emph{Fusion-Invariant Representation for Symmetric Groups} (2023) Bulletin of the Iranian Mathematical Society \textbf{50} doi:10.1007/s41980-024-00867-y
\bibitem{navarro} G. Navarro \emph{Characters and Blocks of Finite Groups} (1998) Cambridge University Press ISBN  978-0-521-59513-1 
\bibitem{sambale}B. Sambale \emph{Fusion Invariant Characters of $p$-groups} (2024), doi:10.48550/arXiv.2401.15706 
\bibitem{realising} S. Park \emph{Realizing a Fusion System by a Single Finite Group} (2010) Arch. Math (Basel), \textbf{94}, pp. 405-410. doi:10.1007/s00013-010-0119-z
\bibitem{serre} J.-P. Serre \emph{Linear Representations of Finite Groups} (1977), Springer Graduate texts in Mathematics No.42, 1st ed. ISBN 978-1-4684-9458-7
\bibitem{isaacs} I.M. Isaacs \emph{Characters of Solvable Groups} (2018), American Mathematical Sociely Graduate Studies in Mathematics No. 189, 1st ed. ISBN 978-1-4704-3485-4
\bibitem{olsson} J. Olsson \emph{Aus dem Nachlass Von Richard Brauer} (1982) Unpublished notes.
\bibitem{geometry of numbers} J.W.S. Cassels \emph{An Introduction to the Geometry of Numbers} (1996) Springer Berlin ISBN 978-3-540-61788-4
\bibitem{aycin thesis} A. $\mathrm{\dot{I}}$plik\c{c}i \emph{Mackey Functors and Fusion Systems} (2017) $\mathrm{\dot{I}}$stanbul Bilgi Uni. Masters Thesis. Avaliable at https://digitalarchive.library.bogazici.edu.tr/server/api/core/bitstreams/93c0c7e2-db9e-4fbe-beb3-49481b455cac/content
\bibitem{martinet} J. Martinet \emph{Perfect Lattices in Euclidean Spaces} (2003), Springer Berlin, 1st ed. ISBN 978-3-540-44236-3
\end{thebibliography}
\end{document}